\documentclass[reqno]{amsart}
\usepackage{amssymb}
\usepackage{graphicx}
\usepackage{amsthm}
\usepackage{amsfonts}
\usepackage{latexsym}
\usepackage{amsmath}
\usepackage{setspace}

\newtheorem{theorem}{Theorem}[section]

\newtheorem{proposition}[theorem]{Proposition}
\newtheorem{corollary}[theorem]{Corollary}
\newtheorem{lemma}[theorem]{Lemma}

\theoremstyle{definition}
\newtheorem{definition}[theorem]{Definition}
\newtheorem{remark}[theorem]{Remark}
\newtheorem{example}[theorem]{Example}

\begin{document}
\title[Admissible Banach function spaces and  nonuniform
stabilities]{Admissible Banach function spaces and  nonuniform stabilities}
\author[N. Lupa]{Nicolae Lupa}
\address{N. Lupa, Department of Mathematics, Politehnica University of Timi\c soara,
Victoriei Square 2, 300006 Timi\c soara, Romania}
\email{nicolae.lupa@upt.ro}
\author[L.H. Popescu]{Liviu Horia Popescu}
\address{L. H. Popescu, Department of Mathematics and Informatics, Faculty of Sciences,
University of Oradea, Universit\u a\c tii St. 1, 410087 Oradea, Romania}
\email{lpopescu2002@yahoo.com}

\begin{abstract}
For nonuniform  exponentially bounded evolution families defined on Banach spaces,
we introduce a class of  Banach function spaces, whose norms are  completely
determined by the nonuniform behaviour of the corresponding  evolution
family.
We generalize a classical theorem of Datko on  these spaces. In addition,
we obtain new criteria for the existence of nonuniform stability.
\end{abstract}

\maketitle

\section{Introduction}

In a recent paper, the authors introduced a special class of  Banach function spaces in order to characterize the concept of nonuniform exponential stability for  evolution families in terms of invertibility of the infinitesimal generators of certain $C_0$-semigroups  \cite{Lu.Po}.
In fact, the results obtained in \cite{Lu.Po} prove the necessity
of the study of nonuniform behaviour of an evolution family for each individual
admissible exponent, requiring substantial departure from well-known ideas.

The aim of this paper is twofold. On the one hand, we try to give some new criteria for the existence of nonuniform stability, which is the case when 0 is an admissible exponent. So, we define some appropriate  projections and  analyse their invariance in respect to  a class of admissible Banach function spaces. On the other hand, motivated by a recent work of Dragi\v{c}evi\'{c} \cite{Dr2},  we intend to generalize a classical result of Datko on these spaces. The reader will surely notice that our techniques are of a completely different type from those in \cite{Dr2}, where the results in the case of continuous time are obtained by reducing the dynamics to the case of discrete time (see also Remark \ref{rem2}).

A notable result in the stability theory of ODEs lies in Datko's theorem.
In order to extend a classical theorem of  Lyapunov to strongly continuous
semigroups of operators in Hilbert spaces,  Datko proved that if
$\{T(t)\}_{t\geq0}$ is a $C_{0}$-semigroup on a complex Hilbert space $X$ such
that $\int_{0}^{\infty} \parallel T(t)x\parallel^{2} dt<\infty$ for some $x\in
X$, then $\lim\limits_{t\to\infty} T(t)x=0$ (Lemma 3 in \cite{Da.1970}).
Later, using different techniques, Pazy showed that if for some
$p\in[1,\infty)$, the integral $\int_{0}^{\infty} \parallel T(t)x\parallel^{p}
dt$ is finite for every $x\in X$, then $\{T(t)\}_{t\geq0}$ is uniform
exponentially stable, i.e. there exist $M,\alpha>0$ such that $\parallel
T(t)\parallel\leq M e^{-\alpha t}$, $t\geq0$ (see \cite{Pa.1972} and
\cite[Theorem 4.1, pp. 116]{Pa.1983}). In 1972 Datko generalized this
result for evolution families on Banach spaces \cite{Da.1972}. It is
shown that a uniform exponentially bounded evolution family $\left\{
U(t,s)\right\}_{t\geq s\geq0}$ on a Banach space $X$ is uniform
exponentially stable if and only if there exists $p\in\lbrack1,\infty)$ such
that
\[
\sup\limits_{t\geq0}\int_{t}^{\infty}\parallel U(\xi,t)x\parallel^{p}%
d\xi<\infty,\;x\in X.
\]

Ichikawa obtained in \cite{Ic.1984} a Datko-type theorem for families of
nonlinear operators $U(t,s):X_{s}\rightarrow X_{t}$, $t\geq s\geq0$, which satisfy the following conditions:

\begin{enumerate}
\item[($e_{1}$)] $U(t,t)x=x$, $x\in X_{t}$;

\item[($e_{2}$)] $U(t,\tau)U(\tau,s)=U(t,s)$ on $X_{s}$, $t\geq\tau\geq s$;

\item[($e_{3}$)] the map $t\mapsto U(t,s)x$ is continuous on $[s,\infty)$ for
every $x\in X_{s}$.
\end{enumerate}
In fact, Ichikawa proved that if there exists a continuous function
$g:\mathbb{R}_{+}\rightarrow(0,\infty)$ such that $\parallel U(t,s)x\parallel
\leq g(t-s)\parallel x\parallel$, for  $x\in X_{s}$ and  $t\geq s\geq0$, then
$\{U(t,s)\}_{t\geq s\geq0}$ is uniform exponentially stable if and only if
there exist  $p,K>0$ such that
\[
\int_{t}^{\infty}\parallel U(\xi,t)x\parallel^{p}d\xi\leq K \parallel
x\parallel^{p},\;x\in X_{t},\,t\geq0.
\]
Using different techniques, this result has been generalized for weak
exponential stability in \cite{Lu.Me.Po}. An interesting Datko-type theorem
was obtained by  Preda and  Megan for uniform exponential dichotomy
\cite{Pr.Me.1985}. Some nonlinear versions have been considered by  Sasu and
 Sasu for both exponential dichotomy \cite{Sa.2010-1} and exponential
trichotomy \cite{Sa.Sa.2011}.

It is worth to mention that the claim in Datko's theorem can be obtained as a
consequence of some results in the evolution semigroup theory (ex.
\cite[Theorem 3.23, pp. 74-75]{Ch.La.1999}).

In the last decade the research in the
field of evolution families developed in a more general direction, which is
the nonuniform behaviour.  A major motivation for
considering this type of behaviour lies in the ergodic theory (for more details we refer the reader to
\cite{Ba.Pe,Ba.Va.2008-1}).
The most recent extensions of Datko's theorem deal with the nonuniform behaviour
(see, for instance, papers
\cite{Be.Lu.Me.Si,Dr1,Dr2,Lu.Me.2014,Pr.Pr.Cr.2012}). We
mention in particular a recent and interesting work of Dragi\v{c}evi\'{c} \cite{Dr2},
where the author obtains some Datko-type characterizations of strong
nonuniform exponential behaviour (contractions, expansions and dichotomies), considering both the case of discrete time as
well as the case of continuous time. For example (see Theorem 9 in \cite{Dr2}), Dragi\v{c}evi\'{c} proves that an invertible evolution family
$\left\{U(t,s)\right\}_{t, s\geq0}$ on a Banach space $X$  admits a {strong nonuniform exponential contraction} if and only if there exists a family $\left\{\|\cdot\|_{t}\right\}_{t\geq 0}$ of norms in $X$ such that:
\begin{enumerate}
\item there exist $C>0$ and $\varepsilon\geq 0$ such that
$$\|x\|\leq \|x\|_t\leq C e^{\varepsilon t} \|x\|, \text{ for every } t\geq 0 \text{ and } x\in X;$$
\item the map $t\mapsto \|x\|_t$ is measurable on $[0,\infty)$ for each $x\in X$;
\item there exist $M, a>0$ such that
$$\|U(t,s)x\|_t\leq M e^{a |t-s|} \|x\|_s, \text{ for } t,s\geq 0 \text{ and } x\in X;$$
\item there exist $p,K>0$  such that
$$\int_{t_0}^\infty \| U(t,t_0)x\|^p_t \,dt\leq K \|x\|^p_{t_0}, \text{ for all } t_0\geq 0 \text{ and } x\in X.$$
\end{enumerate}

In this paper we introduce a family of projections on the space of all
continuous functions $u:\mathbb{R}_{+}\rightarrow X$ and connect the existence of nonuniform (exponential) stability to the restriction of these projections on a class of admissible Banach function spaces, generalizing Datko's theorem on these spaces.

\section{Preliminaries}

\label{sec.preliminaries}

In this section we set up notations and terminology, and  present some auxiliary results  needed in the subsequent part of the paper.

Throughout our paper  $X$ is a Banach space and  $\mathcal{L}(X)$  denotes the Banach algebra of all bounded linear operators on $X$.
Furthermore, $C(\mathbb{R}_{+},X)$ is the space of all
continuous functions $u:\mathbb{R}_{+}\rightarrow X$, $C_{b}(\mathbb{R}%
_{+},X)$ denotes the subspace of the bounded maps in $C(\mathbb{R}_{+},X)$, and
$C_{c}(\mathbb{R}_{+},X)$ stands for the set of all maps in $C(\mathbb{R}_{+},X)$ with
compact support.

\begin{definition}
\label{d.ev} A family $\mathcal{U}=\left\{  U(t,s)\right\}  _{t\geq s\geq0}$
of operators in $\mathcal{L}(X)$ is called an \emph{evolution family} if the
following conditions hold:

\begin{enumerate}
\item $U(t,t)=\mathrm{Id}$, for $t\geq0$;

\item $U(t,\tau)U(\tau,s)=U(t,s)$, for $t\geq\tau\geq s\geq0$;

\item the map $(t,s)\mapsto U(t,s)x$ is continuous on $\{(t,s)\in
\mathbb{R}_{+}^{2}:\,t\geq s\}$ for each $x\in X$.
\end{enumerate}
\end{definition}

\begin{definition}
\label{d.exp.bound} Let $\alpha\in\mathbb{R}$. An evolution family
$\mathcal{U}$ is said to be \emph{$\alpha$-nonuniform exponentially bounded}
if there exists a continuous function $M:\mathbb{R}_{+} \rightarrow(0,\infty)$
such that
\begin{equation}
\parallel U(t,s)\parallel\leq M(s)e^{\alpha(t-s)}, \quad t\geq s\geq0.
\label{eq.exp.bound}%
\end{equation}

If inequality \eqref{eq.exp.bound} holds for some $\alpha<0$, then the
evolution family $\mathcal{U}$ is called \emph{$\alpha$-nonuniform
exponentially stable}. Each $\alpha$ satisfying \eqref{eq.exp.bound} is called
an \emph{admissible exponent} for $\mathcal{U}$, and $\mathcal{A}%
(\mathcal{U})$ denotes the set of all admissible exponents for $\mathcal{U}$.
\end{definition}

Evidently, if $\alpha\in\mathcal{A}(\mathcal{U})$ and $\beta\geq \alpha$, then $\beta\in\mathcal{A}(\mathcal{U})$ and thus the set $\mathcal{A}(\mathcal{U})$ is either a semi-infinite interval, or empty.

An evolution family $\mathcal{U}$ is called \emph{nonuniform exponentially bounded} if $\mathcal{A}(\mathcal{U})\neq\emptyset$. In particular,
if $0\in\mathcal{A}\left(  \mathcal{U}\right) $,
$\mathcal{U}$ is called \emph{nonuniformly stable},
and if $\mathcal{A}(\mathcal{U})$
contains negative admissible exponents, $\mathcal{U}$ is said to be
\emph{nonuniform exponentially stable}.

\medskip

In the above-mentioned terminology, whenever
there exists a \emph{bounded}  function $M:\mathbb{R}_{+} \rightarrow(0,\infty)$
satisfying \eqref{eq.exp.bound}, which is equivalent to the existence of a
positive constant $M>0$ such that
\[
\parallel U(t,s)\parallel\leq Me^{\alpha(t-s)}, \; t\geq s\geq0,
\]
we just replace the term \textquotedblleft nonuniform\textquotedblright\ with
\textquotedblleft uniform\textquotedblright.

\medskip

Throughout our paper we assume that $\mathcal{A}(\mathcal{U})\neq\emptyset$, that is the evolution family
$\mathcal{U}$ is nonuniform exponentially bounded.

For $\alpha\in
\mathcal{A}(\mathcal{U})$, $t\geq0$ and $u\in{C}(\mathbb{R}_{+},X)$, we set
\begin{equation}
\varphi_{\alpha}(t,u)=\sup_{\tau\geq t}e^{-\alpha(\tau-t)}\parallel
U(\tau,t)u(t)\parallel. \label{eq.fi}%
\end{equation}

It is easily seen that
\begin{equation}
\parallel u(t)\parallel\leq\varphi_{\alpha}(t,u)\leq M(t)\parallel
u(t)\parallel. \label{eq.init}%
\end{equation}

The following result can be found in  \cite{Lu.Po} and  assures the
existence of the integrals in the next section.

\begin{proposition}
\label{p2} Let $\alpha\in\mathcal{A}(\mathcal{U})$. The function
$\mathbb{R}_{+}\ni t\mapsto\varphi_{\alpha}(t,u)\in\mathbb{R}_{+}$ is
continuous for every $u\in{C}(\mathbb{R}_{+},X)$.
\end{proposition}

For each $\alpha\in\mathcal{A}\left(  \mathcal{U}\right)  $ we consider the set
\[
\mathcal{C}_{b}(\mathcal{U},\alpha)=\left\{  u\in C(\mathbb{R}_{+},X):\,\sup\limits_{t\geq
0}\varphi_{\alpha}(t,u)<\infty\right\}  .
\]

Since $C_{c}(\mathbb{R}_{+},X)\subset\mathcal{C}_{b}(\mathcal{U},\alpha)$, it follows that this set is nonempty. Moreover, it is easy to check that
$\mathcal{C}_{b}(\mathcal{U},\alpha)$ is a Banach space equipped with the norm
\[
\parallel u\parallel_{\alpha}=\sup\limits_{t\geq0}\text{ }\varphi_{\alpha
}(t,u),
\]
and we call it the \emph{admissible Banach function space} corresponding to
the evolution family $\mathcal{U}$
and the admissible exponent $\alpha\in\mathcal{A}\left(  \mathcal{U}\right)  $.

From (\ref{eq.init}) we have $\mathcal{C}_{b}(\mathcal{U},\alpha)\subset
C_{b}(\mathbb{R}_{+},X).$ Furthermore, following the same line as in the proof
of Proposition 2.6 in \cite{Lu.Po} we deduce that \emph{the Banach spaces
$\mathcal{C}_{b}(\mathcal{U},\alpha)$ and $C_{b}(\mathbb{R}_{+},X)$ coincide if and only if
$\mathcal{U}$ is $\alpha$-uniform exponentially bounded}.

The following example is taken from \cite{Lu.Po}, and proves the necessity of
the study of nonuniform behaviour of an evolution family $\mathcal{U}$ for each
individual admissible exponent $\alpha\in\mathcal{A}(\mathcal{U})$.

\begin{example}
\label{ex1}For $t\geq s>0$ we set
\[
f\left(  t,s\right)  =s\left(  \sqrt{2}+\sin\ln s\right)  -t\left(  \sqrt
{2}+\sin\ln t\right)  ,
\]
$f(t,0)=-t\left(  \sqrt{2}+\sin\ln t\right)  $ for $t>0$, and $f(0,0)=0$.

The
evolution family $\mathcal{U}=\left\{  U(t,s)\right\}_{t\geq s\geq0}$, where $U\left(  t,s\right)  =e^{f\left(  t,s\right)  }{Id},$
has the following properties:

\begin{enumerate}
\item $\mathcal{A}\left(  \mathcal{U}\right)  =\left[  1-\sqrt{2}%
,\infty\right)  $;

\item If $\alpha\geq 0$, then  $\mathcal{U}$ is $\alpha$-uniform exponentially bounded, and hence $$\mathcal{C}_{b}(\mathcal{U},\alpha)=C_{b}(\mathbb{R}_{+},X)$$ whenever
$\alpha\geq0$;

\item $\mathcal{U}$ is not $\alpha$-uniform exponentially bounded, if
$\alpha\in\left[  1-\sqrt{2},0\right)  $;

\item $\mathcal{C}_{b}(\mathcal{U},\alpha)\neq\mathcal{C}_{b}(\mathcal{U},\beta)$, for all different
$\alpha,\beta\in\lbrack1-\sqrt{2},0)$.
\end{enumerate}
\end{example}

\medskip

For any fixed $t\geq0$ and $u\in C(\mathbb{R}_{+},X)$ we  consider a map
$\Phi (t)u:\mathbb{R}_{+}\rightarrow X$, defined by
\[
\left(  \Phi (t) u\right)  (\xi)=%
\begin{cases}
U(\xi,t)u(t), & \xi\geq t,\\
u(\xi), & 0\leq\xi<t.
\end{cases}
\]

Obviously, the operator $\Phi (t): C(\mathbb{R}_{+},X) \to C(\mathbb{R}_{+},X)$ is well-defined, and for all $t\geq0$ and $u\in
C(\mathbb{R}_{+},X)$ we have
\begin{equation}
\varphi_{\alpha}(t,\Phi(t)u)=\varphi_{\alpha}(t,u). \label{eq.1}%
\end{equation}

\begin{lemma}
\label{p.proiector} One has
\begin{equation}
\Phi(t)\Phi(s)=\Phi(s)\Phi(t)=\Phi(s),  \text{ for all } t\geq s\geq 0.\label{eq.pre.D}%
\end{equation}
In particular,  $\Phi(t)$  is a projection on $C(\mathbb{R}_{+},X)$ for each $t\geq0$.
\end{lemma}

\begin{proof}
Let $t\geq s\geq0$ and $u\in C(\mathbb{R}_{+},X)$. For $\tau\geq t$ we get
\begin{align*}
\left(  \Phi(t)(\Phi(s)u)\right)  (\tau)  &  =U(\tau,t)(\Phi(s)u)(t)=U(\tau
,t)U(t,s)u(s)\\
&  =U(\tau,s)u(s)=(\Phi(s)u)(\tau).
\end{align*}
On the other hand, if $\tau<t$ it follows that $\left(  \Phi(t)(\Phi(s)u)\right)  (\tau
)=(\Phi(s)u)(\tau).$ The above identities imply $\Phi(t)\Phi(s)=\Phi(s)$.
The second identity in \eqref{eq.pre.D} can be easily proved using the same
type of arguments and we leave it to the reader. Setting  $t=s$ in \eqref{eq.pre.D}, we have that $\Phi
^{2}(t)=\Phi(t)$, that is $\Phi(t)$ is a projection on $C(\mathbb{R}_{+},X)$.

\end{proof}

\begin{lemma}
\label{lem.Datko} Let $\alpha\in\mathcal{A}\left(  \mathcal{U}\right)  $. For
each $u\in C(\mathbb{R}_{+},X)$ one gets
\begin{equation}
\varphi_{\alpha}(t,\Phi(s)u)\leq e^{\alpha(t-\xi)}\varphi_{\alpha}(\xi
,\Phi(s)u), \text{ for all }t\geq\xi\geq s\geq0. \label{eq.D}%
\end{equation}
In particular,
\begin{equation}
\varphi_{\alpha}(t,\Phi(s)u)\leq e^{\alpha(t-s)}\varphi_{\alpha}(s,u),\text{ for all } t\geq
s\geq0. \label{eq.D.part}%
\end{equation}
\end{lemma}

\begin{proof}
For $t\geq\xi\geq s\geq0$ we have
\begin{align*}
\varphi_{\alpha}(t,\Phi(s)u)  &  =\sup\limits_{\tau\geq t}e^{-\alpha(\tau
-t)}\parallel U(\tau,t)(\Phi(s)u)(t)\parallel\\
&  =e^{\alpha(t-\xi)}\sup\limits_{\tau\geq t}e^{-\alpha(\tau-\xi)}\parallel
U(\tau,s)u(s)\parallel\\
&  \leq e^{\alpha(t-\xi)}\sup\limits_{\tau\geq\xi}e^{-\alpha(\tau-\xi
)}\parallel U(\tau,\xi)(\Phi(s)u)(\xi)\parallel\\
&  =e^{\alpha(t-\xi)}\varphi_{\alpha}(\xi,\Phi(s)u).
\end{align*}
Setting $\xi=s$ in \eqref{eq.D} and using \eqref{eq.1}, we deduce the
inequality \eqref{eq.D.part}.

\end{proof}

\section{The main results}

This section is devoted to the study of nonuniform stability and nonuniform exponential stability for a nonuniform exponentially bounded  evolution family.

\subsection{Nonuniform stability}

The next theorem gives a complete  characterization of nonuniform stability in terms of the invariance of each projection $\Phi(t)$ in respect to an  admissible Banach function space $\mathcal{C}_{b}(\mathcal{U},\alpha)$, for some admissible exponent $\alpha$.

\begin{theorem}
\label{prop.admisible.nes}   The following conditions are equivalent:
\begin{enumerate}
\item[(i)] $\mathcal{U}$ is nonuniformly stable;

\item[(ii)] For some $\alpha\in\mathcal{A}\left(  \mathcal{U}\right)  $,  one has
\[
\Phi(t)\mathcal{C}_{b}(\mathcal{U},\alpha)\subset\mathcal{C}_{b}(\mathcal{U},\alpha), \text{ for every } t\geq 0;
\]
\item[(iii)] There exists  $\alpha\in\mathcal{A}(\mathcal{U})$ such that
\[
\Phi(t) {C}_{c}(\mathbb{R}_+,X)\subset \mathcal{C}_{b}(\mathcal{U},\alpha), \text{ for every } t\geq 0.
\]
\end{enumerate}
\end{theorem}

\begin{proof}
(i)  $\Rightarrow$ (ii): Fix $t\geq0$ and let $u\in \mathcal{C}_{b}(\mathcal{U},0)$.
If $s\in[0,t)$, then
\begin{equation*}\label{a1}
\varphi_{0}(s,\Phi(t)u)
=\sup\limits_{\xi\geq s}\| U(\xi,s)(\Phi(t)u)(s)\|
=\sup\limits_{\xi\geq s}\| U(\xi,s)u(s)\|=\varphi_{0}%
(s,u)\leq\left\Vert u\right\Vert _{0}.
\end{equation*}
On the other hand, if $s\geq t$, from \eqref{eq.D.part} one gets
\begin{equation*}\label{a2}
\varphi_{0}(s,\Phi(t)u)\leq \varphi_{0}(t,u)\leq\left\Vert
u\right\Vert _{0}.
\end{equation*}
We conclude that
\[
\|\Phi(t)u\|_0=\sup\limits_{s\geq0}\varphi_{0}(s,\Phi(t)u)\leq\left\Vert u\right\Vert _{0},
\]
that is $\Phi(t)u\in\mathcal{C}_{b}(\mathcal{U},0)$, and thus (ii) holds for
$\alpha=0$. The implication (ii) $\Rightarrow$ (iii) is trivial.\\
(iii) $ \Rightarrow$ (i):  Assume that (iii) holds, i.e. there exists $\alpha\in\mathcal{A}\left(  \mathcal{U}\right)  $ such that for every $t\geq 0$ we have
$$\Phi(t)u\in\mathcal{C}_{b}(\mathcal{U},\alpha),\text{ for all } u\in{C}_{c}(\mathbb{R}_+,X).$$
It suffices to assume that $\alpha>0$.
For any fixed  $s\geq0$ and $x\in X$, we pick a map $u=u_{s,x}\in C_{c}%
(\mathbb{R}_{+},X)$ such that $u\left(  s\right)  =x$. For  $t\geq s$ one
gets
\begin{align*}
\varphi_{\alpha}(t,\Phi(s)u)
&=\sup\limits_{\xi\geq t}e^{-\alpha(\xi-t)}\parallel U(\xi,t)(\Phi(s)u)(t)\parallel\\
&=\sup\limits_{\xi\geq t}e^{-\alpha(\xi-t)}\parallel U(\xi,t)U(t,s)u(s)\parallel\\
&=\sup\limits_{\xi\geq t}e^{-\alpha(\xi-t)}\parallel U(\xi,s)x\parallel.
\end{align*}
Since $\Phi(s)u\in \mathcal{C}_{b}(\mathcal{U},\alpha)$, it follows that
$$e^{-\alpha(\xi-t)}\parallel U(\xi,s)x\parallel\leq \varphi_\alpha(t,\Phi(s)u)\leq \|\Phi(s)u\|_\alpha<\infty,$$
for all  $(\xi,t)\in\Delta_s=\{(\xi,t):\,\xi\geq t\geq s\}.$
Uniform boundedness principle implies that  $$K_\alpha(s)=\sup\limits_{(\xi,t)\in\Delta_s}e^{-\alpha(\xi-t)}\parallel U(\xi,s)\parallel<\infty.$$
Since $(t,t)\in\Delta_s$, we deduce that
\[
\left\Vert U\left(  t,s\right)  \right\Vert \leq K_{\alpha}\left(  s\right),\text{ for all }t\geq s,
\]
that is $0\in\mathcal{A}\left(  \mathcal{U}\right)$.

\end{proof}

To the best of authors knowledge the above result provides a new criterion for the existence of nonuniform stability.

\begin{remark}\rm
Since $\Phi(t)$ is a projection on $C(\mathbb{R}_+,X)$ for each $t\geq 0$, the preceding theorem implies that  an  evolution family $\mathcal{U}$ is nonuniformly stable if and only if there exists $\alpha\in\mathcal{A}\left(  \mathcal{U}\right)  $ such that the restriction of the operator $\Phi(t)$ on  $\mathcal{C}_{b}(\mathcal{U},\alpha)$ is a projection on
$\mathcal{C}_{b}(\mathcal{U},\alpha)$ for each  $t\geq 0$.
\end{remark}

For any fixed  $\tau\geq0$ we consider  the integral equation%

\begin{equation}
u\left(  t\right)  =U\left(  t,s\right)  u\left(  s\right)  + \int_s^t U\left(  t,\xi\right)  f\left(  \xi\right)  d\xi,\quad t\geq
s\geq\tau.\label{eq.tau}%
\end{equation}

Closely following the arguments in Lemma 1.3 from
\cite{Mi.Ra.Sc.}, one can  prove the next result:

\begin{lemma}
\label{lemA}
The operators
$$A_{\alpha,\tau}:D\left(  A_{\alpha,\tau
}\right)  \subset\mathcal{C}_{b}(\mathcal{U},\alpha)\rightarrow\mathcal{C}_{b}(\mathcal{U},\alpha),\quad
A_{\alpha,\tau}u=f,$$
where
$D\left(  A_{\alpha,\tau}\right)$ is the set of all $u\in\mathcal{C}_{b}(\mathcal{U},\alpha)$ such that there exists $f\in \mathcal{C}_{b}(\mathcal{U},\alpha)$ with $f(t)=0$, for all $t\in[0,\tau)$, and  $u,f$ satisfy  \eqref{eq.tau},
are well-defined and closed in
the topology of $\mathcal{C}_{b}(\mathcal{U},\alpha)$.
\end{lemma}

The above construction is inspired by the definition of the operator $I_{X}$ in
\cite{Mi.Ra.Sc.}.

For any fixed $\alpha\in \mathcal{A}(\mathcal{U})$   and $\tau\geq 0$, the following equivalences hold:
\begin{align*}
A_{\alpha,\tau}u=0
&\Leftrightarrow  u\in\mathcal{C}_{b}(\mathcal{U},\alpha) \text{ and } u\left(  t\right)=U\left(  t,s\right)  u\left(  s\right), \text{ for all }t\geq s\geq\tau\\
&\Leftrightarrow  u\in\mathcal{C}_{b}(\mathcal{U},\alpha) \text{ and } u\left(  \xi\right)=U\left( \xi,\tau\right)  u\left(  \tau\right), \text{ for all }\xi\geq\tau\\
&\Leftrightarrow  u\in\mathcal{C}_{b}(\mathcal{U},\alpha) \text{ and }\Phi(\tau)u=u,
\end{align*}
that is%
\[
\ker A_{\alpha,\tau}=\Phi(\tau)\mathcal{C}_{b}(\mathcal{U},\alpha)\cap\mathcal{C}_{b}(\mathcal{U},\alpha).
\]

Theorem \ref{prop.admisible.nes} and the above considerations
furnish a new necessary and sufficient condition for the existence of
nonuniform stability:

\begin{corollary}
\label{corA}
An evolution family $\mathcal{U}$ is nonuniformly stable if and only if
there exists $\alpha\in\mathcal{A}\left(\mathcal{U}\right)$ such that
\[
\ker A_{\alpha,\tau}=\Phi(\tau)\mathcal{C}_{b}(\mathcal{U},\alpha), \text{ for every } \tau\geq0.
\]
\end{corollary}

\subsection{Nonuniform exponential stability}

The following result gives a necessary condition for the existence of $\alpha$-nonuniform exponential stability.

\begin{proposition}
\label{th.Datko.nec} If an evolution family $\mathcal{U}$ is $\alpha
$-nonuniform exponentially stable, then for each $p\in(0,\infty)$ there exists $K>0$ such
that
\begin{equation}
\int_{t}^{\infty}\varphi_{\alpha}^{p}(\xi,\Phi(t)u)d\xi\leq K\varphi_{\alpha
}^{p}(t,u),\text{ for all }t\geq0\text{ and }u\in C(\mathbb{R}_{+},X).\label{eq.Datko}%
\end{equation}
\end{proposition}

\begin{proof}
Let $t\geq0$ and $u\in C(\mathbb{R}_{+},X)$. From \eqref{eq.D.part} we have
\[
\int_{t}^{\infty}\varphi_{\alpha}^{p}(\xi,\Phi(t)u)d\xi\leq\int_{t}^{\infty
}e^{p\alpha(\xi-t)}\varphi_{\alpha}^{p}(t,u)d\xi=1/({-p\alpha})\varphi
_{\alpha}^{p}(t,u),
\]
which concludes the proof.

\end{proof}

As in the proof of  (i)  $\Rightarrow$ (ii) in  Theorem \ref{prop.admisible.nes}, it is easy to check that if  an evolution family $\mathcal{U}$ is $\alpha
$-nonuniform exponentially stable, then
\begin{equation}\label{eq.restrict}
\Phi(t)\mathcal{C}_{b}(\mathcal{U},\alpha)\subset\mathcal{C}_{b}(\mathcal{U},\alpha), \text{  for every } t\geq 0,
\end{equation}
and thus
the restriction of the operator $\Phi(t)$ on  $\mathcal{C}_{b}(\mathcal{U},\alpha)$, denoted by $\Phi_\alpha(t)$, is a projection on
$\mathcal{C}_{b}(\mathcal{U},\alpha)$ for each  $t\geq 0$. Hence, from Proposition \ref{th.Datko.nec} we obtain the following necessary condition for the existence of $\alpha$-nonuniform exponential stability:

\begin{corollary}
Let $p\in (0,\infty)$. If $\mathcal{U}$ is $\alpha$-nonuniform exponentially stable, then
\[
\int_{t}^{\infty}\varphi_{\alpha}^{p}(\xi,u)d\xi<\infty, \text{ for all } t\geq 0 \text{ and } u\in\mathrm{Range}\left(\Phi_\alpha(t)\right).
\]
\end{corollary}

We now establish a partial converse of Proposition \ref{th.Datko.nec}, which can be
considered a Datko-type theorem for nonuniform exponential stability.

\begin{theorem}
\label{th.Datko.suf}
If there exist $\alpha\in\mathcal{A}(\mathcal{U})$ and
$p,K>0$ such that
\begin{equation}
\int_{t}^{\infty}\varphi_{\alpha}^{p}(\xi,\Phi(t)u)d\xi\leq K\varphi_{\alpha
}^{p}(t,u),\text{ for all }t\geq0\text{ and }u\in \mathcal{C}_{b}(\mathcal{U},\alpha),\label{eq.Datko.suf}%
\end{equation}
then  $\mathcal{U}$ is nonuniform exponentially stable.
\end{theorem}

\begin{proof}
It suffices to assume that $\alpha\geq0$. For the sake of convenience, we divide the proof of the theorem in four steps.\\
{\bf Step 1}. We first prove that there exists a constant $N>0$ such that
\begin{equation}
\varphi_{\alpha}^{p}(t,\Phi(s)u)\leq N\varphi_{\alpha}^{p}(s,u),
\;u\in\mathcal{C}_{b}(\mathcal{U},\alpha), \; t\geq s\geq0. \label{eq.D1}%
\end{equation}

The arguments in this step closely follow a classical idea of Ichikawa
\cite{Ic.1984}. At least for a better understanding, we prefer exposing all
the details. Let $t\geq s\geq0$. If $t\geq s+1$, then from \eqref{eq.D} and
\eqref{eq.Datko.suf} we have
\begin{align*}
\varphi_{\alpha}^{p}(t,\Phi(s)u)  &  =\int_{t-1}^{t}\varphi_{\alpha}%
^{p}(t,\Phi(s)u)d\xi\leq\int_{t-1}^{t}e^{p\alpha(t-\xi)}\varphi_{\alpha}%
^{p}(\xi,\Phi(s)u)d\xi\\
&  \leq e^{p\alpha}\int_{s}^{\infty}\varphi_{\alpha}^{p}(\xi,\Phi(s)u)d\xi\leq
e^{p\alpha}K\varphi_{\alpha}^{p}(s,u).
\end{align*}
On the other hand, if $s\leq t<s+1$, then equation \eqref{eq.D.part} gives
\[
\varphi_{\alpha}^{p}(t,\Phi(s)u)\leq e^{p\alpha(t-s)}\varphi_{\alpha}%
^{p}(s,u)\leq e^{p\alpha}\varphi_{\alpha}^{p}(s,u).
\]
Consequently, we get \eqref{eq.D1} for  $N=\max\left\{  e^{p\alpha}K,e^{p\alpha}\right\}  $.

\medskip

Let us notice that inequality \eqref{eq.D1} implies
$$\Phi(s)\mathcal{C}_{b}(\mathcal{U},\alpha)\subset\mathcal{C}_{b}(\mathcal{U},\alpha), \text{  for every } s\geq 0.$$
Furthermore, according to Theorem \ref{prop.admisible.nes} we get that $\mathcal{U}$ is nonuniformly stable.\\
{\bf Step 2}. We show that for all non-negative integers $n\in\mathbb{N}$,
\begin{equation}
\varphi_{\alpha}^{p}(t,\Phi(s)u)\frac{(t-s)^{n}}{n!N^{n}}\leq N\varphi
_{\alpha}^{p}(s,u), \; u\in\mathcal{C}_{b}(\mathcal{U},\alpha), \; t>s\geq0. \label{eq.D2}%
\end{equation}

We prove inequality \eqref{eq.D2} by induction on $n$. {Step 1} shows that it already
works for $n=0$. Let us assume that \eqref{eq.D2} holds for some $n\in\mathbb{N}$.
Let $u\in\mathcal{C}_{b}(\mathcal{U},\alpha)$ and $t>s\geq0$. According to \eqref{eq.D2}, for
every $\xi\in[s,t)$ we have
\[
\varphi_{\alpha}^{p}(t,\Phi(\xi)u)\frac{(t-\xi)^{n}}{n!N^{n}}\leq
N\varphi_{\alpha}^{p}(\xi,u).
\]
Replacing $u$ by $\Phi(s)u\in \mathcal{C}_{b}(\mathcal{U},\alpha)$ in the above inequality and using \eqref{eq.pre.D} we obtain
\begin{equation}
\varphi_{\alpha}^{p}(t,\Phi(s)u)\frac{(t-\xi)^{n}}{n!N^{n}}\leq N\varphi
_{\alpha}^{p}(\xi,\Phi(s)u). \label{eq.D3}%
\end{equation}
From \eqref{eq.D3} and \eqref{eq.Datko.suf} we get
\begin{align*}
\varphi_{\alpha}^{p}(t,\Phi(s)u) \frac{(t-s)^{n+1}}{(n+1)!N^{n+1}}  &
=\frac{1}{N}\int_{s}^{t}\varphi_{\alpha}^{p}(t,\Phi(s)u)\frac{(t-\xi)^{n}%
}{n!N^{n}}d\xi\leq\int_{s}^{t}\varphi_{\alpha}^{p}(\xi,\Phi(s)u)d\xi\\
&  \leq\int_{s}^{\infty}\varphi_{\alpha}^{p}(\xi,\Phi(s)u)d\xi\leq
K\varphi_{\alpha}^{p}(s,u)\leq N\varphi_{\alpha}^{p}(s,u),
\end{align*}
and so  (\ref{eq.D2}) holds for $n+1$. \\
{\bf Step 3}. For any fixed $\delta\in(0,1)$, there exists a constant
$\widetilde{N}>0$ such that
\begin{equation}
\parallel U(t,s)u(s)\parallel\leq\widetilde{N}e^{-\frac{\delta}{pN}%
(t-s)}\varphi_{\alpha}(s,u), \; u\in\mathcal{C}_{b}(\mathcal{U},\alpha), \; t\geq s\geq0.
\label{eq.D4}%
\end{equation}
Let $u\in\mathcal{C}_{b}(\mathcal{U},\alpha)$, $t>s\geq0$ and pick $\delta\in(0,1)$.
Multiplying inequality \eqref{eq.D2} by $\delta^{n}$ and summing over
$n\in\mathbb{N}$, one gets
\[
\varphi_{\alpha}^{p}(t,\Phi(s)u)\sum\limits_{n=0}^{\infty}\frac{1}{n!}\left[
\frac{\delta(t-s)}{N}\right]  ^{n}\leq N\varphi_{\alpha}^{p}(s,u)\sum
\limits_{n=0}^{\infty}\delta^{n},
\]
which is equivalent to
\[
\varphi_{\alpha}(t,\Phi(s)u)\leq\left(  \frac{N}{1-\delta}\right)
^{1/p}e^{-\frac{\delta}{pN}(t-s)}\varphi_{\alpha}(s,u).
\]
Since
\[
\parallel U(t,s)u(s)\parallel\leq\sup_{\tau\geq t}e^{-\alpha(\tau-t)}\parallel
U(\tau,s)u(s)\parallel=\varphi_{\alpha}(t,\Phi(s)u),\;t\geq s\geq0,
\]
we deduce that (\ref{eq.D4}) works with $\widetilde{N}=1+\left(
{N}/{(1-\delta)}\right)  ^{1/p}$.\\
{\bf Step 4}. For any fixed $\delta\in(0,1)$, the following estimation
holds:
\begin{equation}
\parallel U(t,s)\parallel\leq\widetilde{N}M(s)e^{-\frac{\delta}{pN}%
(t-s)}, \; t\geq s\geq0. \label{eq.D5}%
\end{equation}

For any  fixed $x\in X$ and $s\geq 0$, we consider a map $u_{s,x}\in {C}_{c}(\mathbb{R}_+,X)$ such that $u_{s,x}(s)=x$.
From {Step 3}  and inequalities \eqref{eq.init} we have that for any $\delta\in (0,1)$ there exists $\widetilde{N}>0$ such that
$$\| U(t,s)x\|
=\|U(t,s)u_{s,x}(s)\|
\leq \widetilde{N}e^{-\frac{\delta}{pN}(t-s)}\varphi_{\alpha}(s,u_{s,x})
\leq \widetilde{N}M(s)e^{-\frac{\delta}{pN}(t-s)}\|x\|,$$
for all $t\geq s\geq 0$, which proves that $-\frac{\delta}{pN}\in \mathcal{A}(\mathcal{U})$, thus the evolution family
$\mathcal{U}$ is nonuniform exponentially stable.

\end{proof}

At this point we are able to give a necessary and sufficient condition for the  existence of nonuniform exponential stability.
\begin{corollary}
An evolution family $\mathcal{U}$ is nonuniform exponentially stable if and only if
there exist $\alpha\in\mathcal{A}(\mathcal{U})$ and $p,K>0$ such that
\begin{equation*}
\int_{t}^{\infty}\varphi_{\alpha}^{p}(\xi,\Phi(t)u)d\xi\leq K\varphi_{\alpha
}^{p}(t,u),\text{ for all }t\geq0\text{ and }u\in \mathcal{C}_{b}(\mathcal{U},\alpha).\label{eq.Datko.final}
\end{equation*}
\end{corollary}
The above result is the analog of Theorem 9 in \cite{Dr2} in our context.

\begin{remark}
\label{rem2}From the above arguments (see {Step 4} in the proof of Theorem \ref{th.Datko.suf}) we deduce that
$\mathcal{C}_{b}(\mathcal{U},\alpha)$ can be replaced by any  Banach  space of continuous functions $u:\mathbb{R}_+\to X$,  containing  $C_c(\mathbb{R}_{+},X)$.
Let us notice that, in fact,  the paper \cite{Dr2} treated the
case of constant functions $u:\mathbb{R}_+\to X$ (identified with $X$).
Since $C_c(\mathbb{R}_{+},X)$ is not contained in the space of constant functions $u:\mathbb{R}_+\to X$,
we emphasize that our main results do not imply those in \cite{Dr2}, as well as the ones in \cite{Dr2} do not imply ours.
\end{remark}

\section*{Acknowledgments}

The first author was supported by the research grant GNaC2018-ARUT, no. 1360/01.02.2019.

\end{document}